\documentclass{amsart}

\usepackage{amsmath,amssymb,amsthm}
\usepackage{url}
\usepackage{cite}

\newtheorem{thm}{Theorem}[section]
\newtheorem{cor}[thm]{Corollary}
\newtheorem{lem}[thm]{Lemma}

\theoremstyle{remark}
\newtheorem{rem}[thm]{Remark}
\numberwithin{equation}{section}

\title[Superconvexity of the Heat Kernel on Hyperbolic Space]{Superconvexity of the Heat Kernel on Hyperbolic Space with Applications to Mean Curvature Flow}
\author{Yongzhe Zhang}
\address{Department of Mathematics, California Institute of Technology, 1200 E. California Blvd, Pasadena, CA 91106}
\email{yongzhe@caltech.edu}
\thanks{The author was partially supported by the NSF grants DMS-2018220 and DMS-2018221.} % Do not change the funding information.

\begin{document}

\begin{abstract}
We prove a conjecture of Bernstein that the superconvexity of the heat kernel on hyperbolic space holds in all dimensions and, hence, there is an analog of Huisken's monotonicity formula for mean curvature flow in hyperbolic space of all dimensions. 
\end{abstract}

\maketitle

\section{Introduction} \label{IntroSec}
Throughout the paper, let $\mathbb{H}^n$ be the hyperbolic space of dimension $n$ and let $H_n(t,p;t_0,p_0)$ be the heat kernel on $\mathbb{H}^n$ with singularity at $p=p_0$ at time $t=t_0$. Thus, $H_n$ is the unique positive solution to 
$$
\left\{\begin{array}{ll} \left(\frac{\partial}{\partial t}-\Delta_{\mathbb{H}^n}\right) H_n=0 & \mbox{for $t>t_0$,} \\ \lim_{t\downarrow t_0} H_n=\delta_{p_0}. & \end{array}\right.
$$
By the symmetries of $\mathbb{H}^n$, there is a positive function $K_n(t,\rho)$ on $(0,\infty)\times(0,\infty)$ such that 
$$
H_n(t,p;t_0,p_0)=K_n(t_0-t,\mathrm{dist}_{\mathbb{H}^n}(p,p_0))>0
$$
where $\rho=\mathrm{dist}_{\mathbb{H}^n}(p,p_0)$ is the hyperbolic distance between $p$ and $p_0$. As remarked in \cite{Bernstein}, although $K_n$ can be explicitly computed, the formulas become unmanageable for large $n$; see \cite{DM} for more details.

In this short note we use observations from \cite{DM} and \cite{YZh} to prove the following convexity estimate for $K_n$.

\begin{thm} \label{MainThm}
If $\sigma=\cosh\rho$, then $\log K_n$ is superconvex in $\sigma$, i.e., for any $t>0$ and $\rho>0$,
\begin{equation} \label{SuperconvexEqn}
\frac{\partial^2}{\partial\sigma^2}\log K_n=\left(\frac{1}{\sinh\rho}\frac{\partial}{\partial\rho}\right)^2\log K_n>0.
\end{equation}
\end{thm}

Observe that by the chain rule \eqref{SuperconvexEqn} is equivalent to
\begin{equation} \label{EquivSuperconvexEqn}
\partial_\rho^2 \log K_n(t,\rho)-\coth(\rho)\partial_\rho \log K_n(t,\rho)>0.
\end{equation}
In \cite{Bernstein}, Bernstein proved \eqref{EquivSuperconvexEqn} for small $n$ and conjectured it for all $n$. Hence we confirm this conjecture in Theorem \ref{MainThm}. 

We also give an application of Theorem \ref{MainThm} to the mean curvature flow in hyperbolic space. We say an $n$-dimensional submanifold $\Sigma\subset\mathbb{H}^{n+k}$ has \emph{exponential volume growth}, provided that there is a constant $M>0$ and a point $p_0\in\mathbb{H}^{n+k}$ so that for any $R>0$
$$
Vol_{\mathbb{H}^{n+k}}(\Sigma\cap B_R^{\mathbb{H}^{n+k}}(p_0)) \leq M e^{MR}
$$
where $B_R^{\mathbb{H}^{n+k}}(p_0)$ is the (open) geodesic ball in $\mathbb{H}^{n+k}$ centered at $p_0$ with radius $R$. As noted in \cite[Remark 1.2]{Bernstein}, we can use Theorem \ref{MainThm} to extend \cite[Theorem 1.1]{Bernstein}, an analog of Huisken's monotonicity formula \cite{Huisken} for mean curvature flow in hyperbolic space in low dimensions, to higher dimensions.

\begin{cor} \label{MonotoneCor}
If $\{\Sigma_t\}_{t\in [0,T)}$ is a mean curvature flow of $n$-dimensional complete submanifolds $\Sigma_t\subset\mathbb{H}^{n+k}$ that have exponential volume growth, then, for $t_0\in (0,T]$, $p_0\in\mathbb{H}^{n+k}$ and $t\in (0,t_0)$,
$$
\frac{d}{dt} \int_{\Sigma_t} K_n(t_0-t,\mathrm{dist}_{\mathbb{H}^{n+k}}(p,p_0)) \, dVol_{\Sigma_t}(p) \leq 0
$$
and the inequality is strict unless $\Sigma_t$ is a minimal cone over $p_0$.
\end{cor}

\begin{rem}
In \cite{Bernstein}, Bernstein introduced a notion of hyperbolic entropy for submanifolds in hyperbolic space, which is analogous to the one introduced by Colding-Minicozzi for hypersurfaces in Euclidean space \cite{CM}. Using Corollary \ref{MonotoneCor} and some observations of \cite{Bernstein}, one may adapt the arguments of \cite{BWInvent, BWDuke, BWGT, BWIsotopy,KZh,Zhu} to prove that closed hypersurfaces in hyperbolic space with small hyperbolic entropy are simple in various senses.
\end{rem}

\begin{rem}
Another consequence of Corollary \ref{MonotoneCor} is that the second part of \cite[Theorem 1.5]{Bernstein} (i.e., ``If, in addition, $\Sigma$ is minimal and $n<N$...") holds true for all dimensions $n$. Thus there is a natural relationship between the hyperbolic entropy of an asymptotic regular submanifold of hyperbolic space and the conformal volume of its asymptotic boundary, which is analogous to the relationship between the entropy of an asymptotically conical self-expander and the entropy of its asymptotic cone \cite[Lemma 3.5]{BWIMRN}.
\end{rem}

\section{Proof of Theorem \ref{MainThm}} \label{ProofSec}
Set $\sigma=\cosh\rho$. Let
$$
f_1(\rho)=\frac{\rho}{\sinh\rho} \quad\mbox{for $\rho>0$}
$$
and 
$$
f_{l+1}(\rho)=-\frac{df_{l}}{d\sigma}=(-1)^l\frac{d^l f_1}{d\sigma^l} \quad\mbox{for $\rho>0$}.
$$
It is shown in \cite{DM} that 
$$
K_n(t,\rho)=(4\pi t)^{-\frac{n}{2}} e^{-\frac{(n-1)^2}{4}t} e^{-\frac{\rho^2}{4t}} \alpha_n(t,\rho)
$$
and $\alpha_n(t,\rho)$ satisfies the following recurrence relation:
$$
\alpha_n=f_1\alpha_{n-2}-2t\frac{\partial\alpha_{n-2}}{\partial \sigma}.
$$
As $\mathbb{H}^1$ is the one-dimensional Euclidean space, we have
$$
K_1=(4\pi)^{-\frac{1}{2}} e^{-\frac{\rho^2}{4t}},
$$
so $\alpha_1=1$.

Using the definition of $f_l$ and the recurrence relation for $\alpha_n$, Davies and Mandouvalous prove the following properties for $f_l$ and $\alpha_n$ for odd $n$. 

\begin{lem}[{\cite{DM}}] \label{DMLem}
The following is true:
\begin{enumerate}
\item For each $l\geq 1$, $f_l$ is positive and decreasing.
\item For all $m\geq 1$, 
$$
\alpha_{2m+1}=\sum_{i=0}^{m-1} t^i P_{m,i}(f_1,f_2,\dots,f_m)>0
$$ 
where $P_{m,i}$ are all polynomials with nonnegative coefficients.
\end{enumerate}
\end{lem}

We will also need the following fact proved by C. Yu and F. Zhao.
%\footnote{In \cite[Proposition 3.1]{YZh}, the authors state the weak monotonicity, however, their arguments, in particular, \cite[(3.23)]{YZh}, indeed give the strict monotonicity.}. 

\begin{lem}[{\cite[Proposition 3.1]{YZh}}] \label{YZhLem}
For all $l\geq 1$,
$$
\frac{d}{d\sigma}\left(-\frac{f_{l+1}}{f_l}\right)\geq 0.
$$
\end{lem}

We are now ready to prove Theorem \ref{MainThm}.

\begin{proof}[Proof of Theorem \ref{MainThm}]
By \cite[Proposition 2.1]{Bernstein}, it is sufficient to prove the claim for odd $n\geq 3$. To that end, suppose $n=2m+1$ for some $m\geq 1$. We first compute $\partial_\rho \log K_n(t,\rho)$:
\[
\partial_\rho \log K_n
= \frac{\partial_\rho K_n}{K_n}
\]
and
\[
\partial_\rho K_n = (4\pi t)^{-\frac{n}{2}} e^{-\frac{(n-1)^2}{4}t}
\left( -\frac{\rho}{2t} \alpha_n + \partial_\rho \alpha_n \right) e^{-\frac{\rho^2}{4t}}.
\]
Thus
\[
\partial_\rho \log K_n
= \frac{ -\frac{\rho}{2t} \alpha_n + \partial_\rho \alpha_n }{\alpha_n}
= - \frac{\rho}{2t} +  \partial_\rho \log \alpha_n.
\]
Then we differentiate the above identity with respect to $\rho$:
\begin{align*}
\partial_\rho^2 \log K_n 
& = - \frac{1}{2t} + \partial_\rho^2 \log \alpha_n
\end{align*}
Thus, using the chain rule,
\begin{align*}
\sinh^2(\rho) \partial_\sigma^2 \log K_n
& = \partial_\rho^2 \log K_n - \coth(\rho) \partial_\rho \log K_n\\
& = \left( - \frac{1}{2t} + \partial_\rho^2 \log \alpha_n \right)
- \coth(\rho) \left( -\frac{\rho}{2t} + \partial_\rho \log \alpha_n \right)\\
& = \frac{\rho \coth(\rho)-1}{2t} + \left( \partial_\rho^2 \log \alpha_n - \coth(\rho) \partial_\rho \log \alpha_n \right) \\
& = \frac{\rho \coth(\rho)-1}{2t} + \sinh^2(\rho)  \partial_\sigma^2 \log \alpha_n.
\end{align*}
Using $x > \tanh(x)$ for $x>0$, it is easy to see that the first term is always positive and independent of $n$. Therefore, it suffices to show that 
\[
\partial_\sigma^2 \log \alpha_n  \geq 0
\]
for $n = 2m+1$. Since 
\[ 
\partial_\sigma^2 \log \alpha_n = \frac{(\partial_\sigma^2\alpha_n) \alpha_n - (\partial_\sigma\alpha_n)^2 }{ \alpha_n^2 }
\]
it would be sufficient if we proved the following claim:
\[
A_n \stackrel{\rm def}{=}  (\partial_\sigma^2\alpha_n) \alpha_n - (\partial_\sigma\alpha_n)^2 \ge 0 \quad \mbox{for $n=2m+1$.}
\]
To see this, we need to use Lemma \ref{DMLem} and Lemma \ref{YZhLem} to compute the $\sigma$-derivatives of $\alpha_n$ where $n=2m+1$ for some $m\ge 1$. Since
\begin{align*}
\alpha_{2m+1} = \sum_{i=0}^{m-1} t^i P_{m,i}(f_1,\cdots,f_m)
\end{align*}
it follows that
\begin{align*}
A_{2m+1} 
& = \left[ 
\sum_{i=0}^{m-1} t^i P_{m,i}
\right] \left[
\sum_{i=0}^{m-1} t^i \frac{d^2P_{m,i}}{d\sigma^2} 
\right] - \left[
\sum_{i=0}^{m-1} t^i \frac{dP_{m,i}}{d\sigma}\right]^2\\
& = \sum_{i=0}^{2m-2} t^i 
\sum_{\substack{\alpha+\beta=i \\ 0\leq \alpha,\beta \leq m-1}}\left[
P_{m,\alpha}\frac{d^2P_{m,\beta}}{d\sigma}-\frac{dP_{m,\alpha}}{d\sigma} \frac{dP_{m,\beta}}{d\sigma}
\right].
\end{align*}
To show $A_{2m+1}\ge 0$, it is sufficient to show that for each $0 \le i \le 2m-2$
\[
B_{m,i} \stackrel{\rm def}{=} 
\sum_{\substack{\alpha+\beta=i \\ 0\leq \alpha,\beta\leq m-1}} \frac{d^2P_{m,\beta}}{d\sigma}-\frac{dP_{m,\alpha}}{d\sigma} \frac{dP_{m,\beta}}{d\sigma} \ge 0.
\]
By Lemma \ref{DMLem}, we know that $P_{m,r}(y_1,\cdots,y_m)$ is a polynomial in $y_1,\cdots,y_m$ with nonnegative coefficients, so we can assume that 
\[
P_{m,r}(y_1,\cdots,y_m) 
= \sum_{j_1,\cdots,j_m\geq 0} a_{m,r,j_1\cdots j_m} y_1^{j_1} \cdots y_m^{j_m}
\]
where all $a_{m,r,j_1\cdots j_m}\geq 0$ with only finitely many nonzero.
Then, applying chain rule, one gets
\begin{align*}
\frac{dP_{m,r}}{d\sigma}
& = \sum_{j_1,\cdots,j_m} a_{m,r,j_1\cdots j_m\geq 0}
\sum_{s=1}^{m} f_1^{j_1} \cdots \left(j_s f_s^{j_s-1} \frac{df_s}{d\sigma}\right) \cdots f_m^{j_m}\\
& = \sum_{j_1,\cdots,j_m\geq 0} a_{m,r,j_1\cdots j_m}
f_1^{j_1} \cdots f_m^{j_m}
\left( \sum_{s=1}^{m}  -j_s \frac{f_{s+1}}{f_s} \right) 
\end{align*}
and from $P_{m,r} \geq 0$ and Lemma \ref{YZhLem}, one gets
\begin{align*}
\frac{d^2P_{m,r}}{d\sigma^2}
& = \sum_{j_1,\cdots,j_m\geq 0} a_{m,r,j_1\cdots j_m}
f_1^{j_1} \cdots f_m^{j_m} \left( \sum_{s=1}^{m}  -j_s \frac{f_{s+1}}{f_s} \right)^2\\
& \quad + \sum_{j_1,\cdots,j_m\geq 0} a_{m,r,j_1\cdots j_m}
f_1^{j_1} \cdots f_m^{j_m} \frac{d}{d\sigma}\left( \sum_{s=1}^{m}  -j_s \frac{f_{s+1}}{f_s} \right)\\
& \ge \sum_{j_1,\cdots,j_m\geq 0} a_{m,r,j_1\cdots j_m}
f_1^{j_1} \cdots f_m^{j_m} \left( \sum_{s=1}^{m} -j_s\frac{f_{s+1}}{f_s} \right)^2.
\end{align*}
Now, we realize that we can symmetrize the expression of $B_{m,i}$:
\begin{align*}
2 B_{m.i}
& = \sum_{\substack{\alpha+\beta = i \\ 0\leq \alpha,\beta\leq m-1}} P_{m,\alpha} \frac{d^2P_{m,\beta}}{d\sigma^2}
+ P_{m,\beta}\frac{d^2P_{m,\alpha}}{d\sigma^2}
- 2\frac{dP_{m,\alpha}}{d\sigma} \frac{dP_{m,\beta}}{d\sigma}
\end{align*}
and from the previous computation we know that
\begin{align*}
& C_{m,i,\alpha,\beta} \stackrel{\rm def}{=}
P_{m,\alpha} \frac{d^2P_{m,\beta}}{d\sigma^2}
+ P_{m,\beta}\frac{d^2P_{m,\alpha}}{d\sigma^2}
- 2\frac{dP_{m,\alpha}}{d\sigma} \frac{dP_{m,\beta}}{d\sigma}
\\
& \ge 
\left( \sum_{j_1,\cdots,j_m\geq 0} a_{m,\alpha,j_1\cdots j_m}
f_1^{j_1} \cdots f_m^{j_m} \right)
\left( \sum_{k_1,\cdots,k_m\geq 0} a_{m,\beta,k_1\cdots k_m}
f_1^{k_1} \cdots f_m^{k_m} \left( \sum_{s=1}^{m}  -k_s \frac{f_{s+1}}{f_s} \right)^2 \right)\\
& \quad + \left( \sum_{j_1,\cdots,j_m\geq 0} a_{m,\alpha,j_1\cdots j_m}
f_1^{j_1} \cdots f_m^{j_m}
\left( \sum_{s=1}^{m}  -j_s \frac{f_{s+1}}{f_s} \right)^2  \right)
\left( \sum_{k_1,\cdots,k_m\geq 0} a_{m,\beta,k_1\cdots k_m}
f_1^{k_1} \cdots f_m^{k_m} \right)\\
& \quad - 2 \left( \sum_{j_1,\cdots,j_m\geq 0} a_{m,\alpha,j_1\cdots j_m} f_1^{j_1} \cdots f_m^{j_m}
\left( \sum_{s=1}^{m}  -j_s \frac{f_{s+1}}{f_s} \right) \right)\\
&\quad\quad\quad\quad\times \left( \sum_{k_1,\cdots,k_m\geq 0} a_{m,\beta,k_1\cdots k_m}
f_1^{k_1} \cdots f_m^{k_m} \left( \sum_{s=1}^{m}  -k_s \frac{f_{s+1}}{f_s} \right) \right)\\
& \ge 
\sum_{j_1,\cdots,j_m\geq 0} \sum_{k_1,\cdots,k_m\geq 0}
a_{m,\alpha ,j_1\cdots j_m} a_{m,\beta,k_1\cdots k_m}
f_1^{j_1+k_1} \cdots f_m^{j_m+k_m}\\
&\quad\times \left[
\left( \sum_{s=1}^{m}  -j_s \frac{f_{s+1}}{f_s} \right)^2
+ \left( \sum_{s=1}^{m}  -k_s \frac{f_{s+1}}{f_s} \right)^2
- 2 \left( \sum_{s=1}^{m}  -j_s \frac{f_{s+1}}{f_s} \right)
\left( \sum_{s=1}^{m}  -k_s \frac{f_{s+1}}{f_s} \right)
\right].
\end{align*}
Therefore, by completing squares and the facts that $f_s>0$ and $a_{m,\alpha,j_1\cdots j_m}\geq 0$ and $a_{m,\beta,k_1\cdots k_m}\geq 0$, we conclude $C_{m,i,\alpha,\beta}\geq 0$ and prove the claim.
\end{proof}


\begin{thebibliography}{99}

\bibitem{Bernstein} J. Bernstein, \emph{Colding Minicozzi entropy in hyperbolic Space}, preprint (2020). Available at \url{https://arxiv.org/abs/2007.10218}.

\bibitem{BWInvent} J. Bernstein and L.Wang, \emph{A sharp lower bound for the entropy of closed hypersurfaces up to dimension six}, Invent. Math. 206 (2016), no. 3, 601–-627.

\bibitem{BWDuke} J. Bernstein and L. Wang, \emph{A topological property of asymptotically conical self-shrinkers of small entropy}, Duke Math J 166 (2017), no. 3, 403–-435.

\bibitem{BWGT} J. Bernstein and L. Wang, \emph{Topology of closed hypersurfaces of small entropy}, Geom. Topol. 22 (2018), no. 2, 1109–-1141.

\bibitem{BWIMRN} J. Bernstein and L. Wang, \emph{Smooth compactness for spaces of asymptotically conical self-expanders of mean curvature flow}, IMRN (2019), to appear. Available at \url{https://doi.org/10.1093/imrn/rnz087}.

\bibitem{BWIsotopy} J. Bernstein and L. Wang, \emph{Closed hypersurfaces of low entropy in $\mathbb{R}^4$ are isotopically trivial}, preprint (2020). Available at \url{https://arxiv.org/abs/2003.13858}.

\bibitem{CM} T.H. Colding and W.P. Minicozzi II, \emph{Generic mean curvature flow I; generic singularities}, Ann. of Math. (2) 175 (2012), no. 2, 755–-833.

\bibitem{DM} E.B. Davies and N. Mandouvalos, \emph{Heat kernel bounds on hyperbolic space and Kleinian groups}, Proc. London Math. Soc. (3) 57 (1988), no. 1, 182--208.

\bibitem{KZh} D. Ketover and X. Zhou, \emph{Entropy of closed surfaces and min-max theory}, J. Differential Geom. 110 (2018), no. 1, 31–-71.

\bibitem{Huisken} G. Huisken, \emph{Asymptotic behaviour for singularities of the mean curvature flow}, J. Differential Geom. 31 (1990), no. 1, 285-–299.

\bibitem{YZh} C. Yu and F. Zhao, \emph{Li-Yau multiplier set and optimal Li-Yau gradient estimate on hyperbolic spaces}, preprint (2018). Available at \url{https://arxiv.org/abs/1807.05709}.

\bibitem{Zhu} J. Zhu, \emph{On the entropy of closed hypersurfaces and singular self-shrinkers}, J. Differential Geom. 114 (2020), no. 3, 551–-593.

\end{thebibliography}
\end{document}